\documentclass[10pt,reqno]{amsart}
\usepackage{amssymb}
\usepackage{amsmath}
\usepackage{amsfonts}
\usepackage{hyperref}
\usepackage[title]{appendix}
\usepackage[all,poly]{xy}
\usepackage{enumerate}
\usepackage[square, comma, sort&compress, numbers]{natbib}
\usepackage{extpfeil}
\usepackage{color}
\usepackage{lipsum}
\usepackage{tabularx}
\usepackage{bm}

\allowdisplaybreaks

\begin{document}
    \theoremstyle{plain}
    \newtheorem{thm}{Theorem}[section]
    \newtheorem{theorem}[thm]{Theorem}
    \newtheorem{lemma}[thm]{Lemma}
    \newtheorem{corollary}[thm]{Corollary}
    \newtheorem{corollary*}[thm]{Corollary*}
    \newtheorem{proposition}[thm]{Proposition}
    \newtheorem{proposition*}[thm]{Proposition*}
    \newtheorem{conjecture}[thm]{Conjecture}
    \theoremstyle{definition}
    \newtheorem{construction}[thm]{Construction}
    \newtheorem{notations}[thm]{Notations}
    \newtheorem{question}[thm]{Question}
    \newtheorem{problem}[thm]{Problem}
    \newtheorem{remark}[thm]{Remark}
    \newtheorem{remarks}[thm]{Remarks}
    \newtheorem{definition}[thm]{Definition}
    \newtheorem{claim}[thm]{Claim}
    \newtheorem{assumption}[thm]{Assumption}
    \newtheorem{assumptions}[thm]{Assumptions}
    \newtheorem{properties}[thm]{Properties}
    \newtheorem{example}[thm]{Example}
    \newtheorem{comments}[thm]{Comments}
    \newtheorem{blank}[thm]{}
    \newtheorem{observation}[thm]{Observation}
    \newtheorem{defn-thm}[thm]{Definition-Theorem}
        \newcommand{\Rmnum}[1]{\uppercase\expandafter{\romannumeral #1}}
        \newcommand{\rmnum}[1]{\romannumeral #1}

\def\vol{\operatorname{vol}}


    \title[Super higher Weil-Petersson volumes]{Higher Weil-Petersson volumes of the moduli space of super Riemann surfaces}
    \author{Xuanyu Huang}
    \address{Center of Mathematical Sciences, Zhejiang University, Hangzhou, Zhejiang 310027, China}
    \email{Hxuanyu98@gmail.com}

\author{Kefeng Liu}
        \address{Mathematical Science Research Center, Chongqing University of Technology, Chongqing, 400054, China}
        \email{kefengliu@cqut.edu.cn}

        \author{Hao Xu}
     \address{Mathematical Science Research Center, Chongqing University of Technology, Chongqing, 400054, China}
      \email{mathxuhao@gmail.com}

    \maketitle
\begin{abstract}

Inspired by the theory of JT supergravity, Stanford-Witten derived a
remarkable recursion formula of Weil-Petersson volumes of moduli
space of super Riemann surfaces. It is the super version of the
celebrated Mirzakhani's recursion formula. In this paper, we
generalize Stanford-Witten's formula to include high degree kappa
classes.
\end{abstract}

\section{Introduction}
Two-dimensional quantum gravity has been studied in two different
approaches, one is by matrix models which obeys a particular
integrable hierarchy and the other is the intersection theory on the
moduli space of curves. It was conjectured by Witten
\cite{witten1990two} that these two approaches should be equivalent,
thus the generating function of the intersection numbers of certain
natural cohomology class on the moduli is governed by KdV hierarchy
or equivalently by Virasoro constraints \cite{dijkgraaf1991loop}.
Witten's conjecture was first proved by Kontsevich
\cite{kontsevich1992intersection}. The celebrated Witten-Kontsevich
theorem gives a surprising connection between intersection theory on
the moduli space of curves and integrable system.

Jackiw-Teitelboim (JT) gravity is a simple model of two-dimensional
quantum gravity coupled to the dilaton field. Its recent
applications in SYK model generated a resurgence of interest in JT
gravity, whose partition function is related to the Weil-Petersson
volumes of moduli spaces of Riemann surfaces.

Similarly, as pointed out by Stanford and Witten
\cite{stanford2020jt}, the partition function of JT supergravity is
related to the super Weil-Petersson volumes, i.e. the Weil-Petersson
volumes of moduli spaces of super Riemann surfaces.

Denote by $\hat{V}_{g,n}(L_1,...,L_n)$ the Weil-Petersson volume of
the moduli space of super Riemann surfaces with boundary components
of lengths $L_1,...,L_n$. By definition (see
\cite{stanford2020jt,norbury2020enumerative}), it is the integration
of the Euler form of a holomorphic vector bundle $E_{g,n}^{\vee}$
combined with the Weil-Petersson symplectic form over the moduli
space $\mathcal{M}_{g,n,\vec{o}}^{\rm spin}(L_1,...,L_n)$ of spin hyperbolic
surfaces with boundary geodesics of lengths $L_1,...,L_n$. Here
$\vec{o}=(0,0,...,0)$ represents that all boundary components are
Neveu-Schwarz. That is,
$$\hat{V}_{g,n}(L_1,...,L_n)=\int_{\mathcal{M}_{g,n,\vec{o}}^{\rm spin}(L_1,...,L_n)}e(E_{g,n}^{\vee})\mathrm{exp}(
\omega^{WP}).$$

In a series of innovative papers, Norbury
\cite{norbury2023new,norbury2020enumerative} constructed a
cohomology class $\Theta_{g,n}$ on $\overline{\mathcal{M}}_{g,n}$
and showed that the super Weil-Petersson volume
$\hat{V}_{g,n}(L_1,...,L_n)$ is reduced to the following integral on
ordinary moduli space.
\begin{equation}\label{VTheta}
V^{\Theta}_{g,n}(L_1,...,L_n):=\int_{\overline{\mathcal{M}}_{g,n}}\Theta_{g,n}\exp\left(2\pi^2\kappa_1+\frac{1}{2}\sum_{i=1}^n
L_i^2\psi_i\right).
\end{equation}

Below is Norbury's formula.
\begin{theorem}[Norbury \cite{norbury2020enumerative}]\label{thmNorbury}
\begin{equation}\label{eqNorbury}
\hat{V}_{g,n}(L_1,...,L_n)=2^{1-g-n}V^{\Theta}_{g,n}(L_1,...,L_n).
\end{equation}
\end{theorem}

Now we explain the symbols in \eqref{VTheta}. Denote by
$\overline{\mathcal{M}}_{g,n}$ the moduli space of stable
$n$-pointed genus $g$ Riemann surfaces. We have the morphism that
forgets the last marked point,
$$\pi: \overline{\mathcal{M}}_{g,n+1}\to\overline{\mathcal{M}}_{g,n},$$ and the gluing
maps which glue the last two points, $$\phi_{irr}:
\overline{\mathcal{M}}_{g-1,n+2}\to\overline{\mathcal{M}}_{g,n},$$
$$\phi_{h,I}:\overline{\mathcal{M}}_{h,|I|+1}\times\overline{\mathcal{M}}_{g-h,|J|+1}\to\overline{\mathcal{M}}_{g,n},\qquad
I\sqcup J=\{1,...,n\}.$$ Denote by $\sigma_1,...,\sigma_n$ the
canonical sections of $\pi$, and by $D_1,...,D_n$ the corresponding
divisors in $\overline{\mathcal{M}}_{g,n+1}$. Let $\omega_{\pi}$ be the
relative dualizing sheaf, we have the following tautological classes
on moduli spaces of curves.
$$\psi_i=c_1(\sigma_i^{\ast}(\omega_{\pi})),$$ $$\kappa_i=\pi_{\ast}(c_1(\omega_{\pi}(\sum D_i))^{i+1}).$$
$\Theta$ class is constructed by Norbury\cite{norbury2023new}. It follows that $\Theta_{g,n}\in \mathit{H}^{4g-4+2n}(\overline{\mathcal{M}}_{g,n},\mathbb{Q})$ for $g\geq 0$, $n\geq 0$ and $2g-2+n>0$ such that \\
(\rmnum{1}) $\phi_{irr}^{\ast}\Theta_{g,n}=\Theta_{g-1,n+2}$ and $ \phi_{h,\mathbf{I}}^{\ast}\Theta_{g,n}=\pi_1^{\ast}\Theta_{h,|I|+1}\cdot\pi_2^{\ast}\Theta_{g-h,|J|+1}$,\\
(\rmnum{2})$\Theta_{g,n+1}=\psi_{n+1}\cdot\pi^{\ast}\Theta_{g,n}$,\\ (\rmnum{3}) $\Theta_{1,1}=3\psi_1$ and $\Theta_{0,n}=0$, (Since $\mathrm{deg}\ \Theta_{0,n}=n-2>n-3=\mathrm{dim}(\overline{\mathcal{M}}_{g,n})$)\\where $\pi_i$
is projection onto the $i$-th factor of $\overline{\mathcal{M}}_{h,|I|+1}\times\overline{\mathcal{M}}_{g-h,|J|+1}$.

We are interested in the following intersection numbers, called
super higher Weil-Petersson volumes
$$\langle\kappa(\mathbf{b})\prod\limits_{i=1}\limits^{n}\tau_{d_i}\rangle_{g}^{\Theta}:=\int_{\overline{\mathcal{M}}_{g,n}}\Theta_{g,n}\prod\limits_{i\geq
1}\kappa_{i}^{b(i)}\prod\limits_{i=1}\limits^n \psi_i^{d_i},$$ where
$\sum_{i\geq 1} ib(i)+\sum_{i=1}^{n}d_i=g-1.$ Note that only when
$g\geq 2$, $\kappa(\mathbf{b})$ may not be empty. When
$d_1=\cdots=d_n=0$, denote
$$\setlength{\abovedisplayskip}{3pt}
\setlength{\belowdisplayskip}{3pt}V_{g,n}^{\Theta}(\kappa(\mathbf{b}))=\langle\tau_0^{n}\kappa(\mathbf{b})
\rangle_g^{\Theta}=\int_{\overline{\mathcal{M}}_{g,n}}\Theta_{g,n}\kappa(\mathbf{b}).$$
In particular, $V_{g,0}^{\Theta}(\kappa(\mathbf{b}))$ is simply
denoted by $V_{g}^{\Theta}(\kappa(\mathbf{b}))$.

By adapting the techniques of Mirzakhani
\cite{mirzakhani2007simple}, Stanford and Witten
\cite{stanford2020jt} obtained a super McShane-Mirzakhani identity
via supergeometry and then integrated it to derive a recursion
formula for super Weil-Petersson volumes
$\hat{V}_{g,n}(L_1,...,L_n)$. See \eqref{Volumerecursion} in the
appendix. A detailed treatment of super McShane identity in genus
one was given in \cite{huang2023super}.

From Theorem \ref{thmNorbury}, Norbury \cite{norbury2020enumerative}
gave a new proof of Stanford-Witten formula via topological
recursion and spectral curve.

From \eqref{VTheta} and \eqref{eqNorbury}, Stanford-Witten's formula
is equivalent to the following recursion formula of intersection
numbers.
\begin{equation}\label{Recursionkappsi}
\begin{aligned}&(2d_1+1)!!\langle\kappa_{1}^{a}\prod\limits_{i=1}\limits^{n}\tau_{d_i} \rangle^{\Theta}_g\\&=\sum\limits_{j=2}\limits^{n}\sum\limits_{b=0}\limits^{a}\frac{a!}{(a-b)!}\frac{(2b+2d_1+2d_j+1)!!}{(2d_j-1)!!}\beta_b\langle\kappa_{1}^{a-b}\tau_{b+d_1+d_j}\prod\limits_{i\neq 1,j}\tau_{d_i} \rangle^{\Theta}_g\\&+\frac{1}{2}\sum\limits_{b=0}\limits^{a}\underset{r+s=b+d_1-1}{\sum}\frac{a!}{(a-b)!}(2r+1)!!(2s+1)!!\beta_b\langle\kappa_{1}^{a-b}\tau_r\tau_s\prod\limits_{i=2}\limits^{n}\tau_{k_i} \rangle^{\Theta}_{g-1}\\&+\frac{1}{2}\sum\limits_{b=0}\limits^{a}\sum_{\substack{c+c^{'}=a-b\\[3pt]r+s=b+k_1-1}}\sum_{\substack{g_1+g_2=g\\[3pt]I\sqcup J=\{2,...,n\}}}\frac{a!}{c!c^{'}!}(2r+1)!!(2s+1)!!\beta_{b}\langle\kappa_{1}^{c}\tau_{r}\tau_{d(I)} \rangle^{\Theta}_{g_1}\langle\kappa_{1}^{c^{'}}\tau_{s}\tau_{d(J)} \rangle^{\Theta}_{g_2},\end{aligned}\end{equation} where $\tau_{d(I)}=\prod\limits_{i\in I}\tau_{d_i}$ and $\beta_b$ is defined by \begin{equation}\label{beta_b}\frac{1}{cos(\sqrt{2}x)}=\sum\limits_{b=0}\limits^{\infty}\beta_bx^{2b}.\end{equation}

We remark that \eqref{Recursionkappsi} is also mentioned by Norbury
\cite{norbury2020enumerative} but was not written down explicitly.
For completeness, we give detailed derivation in the appendix (see
the proof of Theorem A.3).

Now we fix notations. Consider the semigroup $N^{\infty}$ of sequences $\mathbf{m}=(m(1),m(2),\\m(3),...)$, where $m(i)$ are nonnegative integers and $m(i)=0$ for sufficiently large $i$. Denote by $\bm{\delta}_{a}$ the sequence with $1$ at the a-th place and zeros else where. Let $\mathbf{m},\mathbf{t},\mathbf{a_1},...,\mathbf{a_n}\in N^{\infty}$, $\mathbf{t}=\sum\limits_{i=1}\limits^{n}\mathbf{a_i}$ and $\mathbf{s}:=(s_1,s_2,...)$ be a family of independent formal variables.
$$|\mathbf{m}|:=\sum_{i\geq 1}im(i),\qquad ||\mathbf{m}||:=\sum_{i\geq 1}m(i),\qquad \mathbf{s}^{\mathbf{m}}:=\prod_{i\geq 1}s_{i}^{m(i)},\qquad \mathbf{m}!:=\prod_{i\geq 1}m(i)!,$$
$$\binom{\mathbf{m}}{\mathbf{t}}:=\prod\limits_{i\geq}\binom{m(i)}{t(i)},\qquad \binom{\mathbf{m}}{\mathbf{a_1},...,\mathbf{a_n}}:=\prod\limits_{i\geq 1}\binom{m(i)}{a_1(i),...,a_n(i)}.$$ Let $\mathbf{b}\in\mathbb{N}^{\infty}$, we denote a formal monomial of $\kappa$ classes by $$\kappa(\mathbf{b}):=\prod\limits_{i\geq 1}\kappa_{i}^{b(i)},$$ and $\mathbf{0}:=(0,0,...)\in N^{\infty}$ as the zero element of $N^{\infty}$.

The Witten-kontsevich theorem asserts that the generating function
for integrals of $\psi$ classes on $\overline{\mathcal{M}}_{g,n}$
\begin{equation}\label{Witten-Kontsevich}Z^{KW}(\hbar,t_0,t_1,...):=\mathrm{exp}\underset{g,n,\vec{d}}{\sum}\frac{\hbar^{g-1}}{n!}\langle\prod\limits_{i=1}\limits^{n}\tau_{d_i} \rangle_{g} \prod\limits_{i=1}\limits^{n}t_{d_i},\end{equation}
is a tau function of the KdV hierachy. Here, we also assemble the
intersection numbers involving $\Theta$ and $\psi$ classes in the
following generating fuction:
\begin{equation}\label{ZTheta}Z^{\Theta}(\hbar,t_0,t_1,...):=\mathrm{exp}\underset{g,n,\vec{d}}{\sum}\frac{\hbar^{g-1}}{n!}\langle\prod\limits_{i=1}\limits^{n}\tau_{d_i} \rangle_{g}^\Theta \prod\limits_{i=1}\limits^{n}t_{d_i}.\end{equation}
Denote by $Z^{BGW}$ the Br{\'e}zin-Gross-Witten tau function of the
KdV hierarchy which is annihilated by Virasoro constraints.
\begin{theorem}\textup{\cite{chidambaram2022relations,norbury2023new}}\label{BGW-Theta} $Z^{\Theta}(\hbar,t_0,t_1,...)$ coincides with the Br{\'e}zin-Gross-Witten tau function of KdV hierarchy. That is
\begin{equation}\label{BGW}
Z^{\Theta}(\hbar,t_0,t_1,...)=Z^{BGW}(\hbar,t_0,t_1,...).
\end{equation}
\end{theorem}

Equation \eqref{BGW} was conjectured by Norbury
\cite{norbury2023new} and proved in \cite{chidambaram2022relations}.
See \cite{dubrovin2021tau,guo2024combinatorics,yang2021gelfand,yang2021extension} for
more studies on BGW tau functions.

From Virasoro constraints (see Section 2), \eqref{BGW} implies the
following DVV type recursion formula.
\begin{equation}\label{Recursionpsi}
\begin{aligned}\langle\tau_{k}\prod\limits_{i=1}\limits^{n}\tau_{d_i} \rangle^{\Theta}_g=\frac{1}{(2k+1)!!}[&\sum\limits_{i=1}\limits^{n}\frac{(2k+2d_i+1)!!}{(2k_i-1)!!}\langle\tau_{k+d_i}\prod\limits_{j\neq i}\tau_{d_j} \rangle^{\Theta}_g\\&+\frac{1}{2}\underset{r+s=k-1}{\sum}(2r+1)!!(2s+1)!!\langle\tau_r\tau_s\prod\limits_{i=1}\limits^{n}\tau_{d_i} \rangle^{\Theta}_{g-1}\\+\frac{1}{2}\underset{r+s=k-1}{\sum}&(2r+1)!!(2s+1)!!\sum_{\substack{g_1+g_2=g\\[3pt]I\sqcup J=\{1,...,n\}}}\langle\tau_{r}\tau_{d(I)} \rangle^{\Theta}_{g_1}\langle\tau_{s}\tau_{d(J)} \rangle^{\Theta}_{g_2}].\end{aligned}\end{equation}
It was first derived by Do and Norbury \cite{do2018topological}. It
is in turn a special case of Stanford-Witten formula
\eqref{Recursionkappsi}.

As main results of our paper, we generalize \eqref{Recursionkappsi}
and \eqref{Recursionpsi} to recursion formulas involving higher
degree $\kappa$ classes.

\begin{theorem}\label{RecursionA}
Let $\mathbf{b}\in\mathbb{N}^{\infty}, g\geq 1, n\geq 1$ and $d_j\geq 0$. Then
\begin{equation}\label{Recursion1}\begin{aligned}&\sum\limits_{\mathbf{L}+\mathbf{L^{'}}=\mathbf{b}}(-1)^{||\mathbf{L}||}\binom{\mathbf{b}}{\mathbf{L}}\frac{(2d_1+2|\mathbf{L}|+1)!!}{(2|\mathbf{L}|-1)!!}\langle\kappa(\mathbf{L^{'}})\tau_{d_1+|\mathbf{L}|}\prod\limits_{j=2}\limits^{n}\tau_{d_j}\rangle_{g}^{\Theta}\\&=\sum\limits_{j=2}\limits^{n}\frac{(2d_1+2d_j+1)!!}{(2d_j-1)!!}\langle\kappa(\mathbf{b})\prod_{i\neq 1,j}\limits^{n}\tau_{d_i}\rangle_{g}^{\Theta}\\&\qquad+\frac{1}{2}\sum\limits_{r+s=d_1-1}(2r+1)!!(2s+1)!!\langle\kappa(\mathbf{b})\tau_{r}\tau_s\prod\limits_{i=2}\limits^{n}\tau_{k_i}\rangle_{g-1}^{\Theta}\\&+\frac{1}{2}\sum_{\substack{\mathbf{e}+\mathbf{f}=b\\[3pt]I\sqcup J=\{2,...,n\}}}\sum\limits_{r+s=d_1-1}\binom{\mathbf{b}}{\mathbf{e}}(2r+1)!!(2s+1)!!\langle\kappa(\mathbf{e})\tau_{r}\tau_{d(I)}\rangle_{g^{'}}^{\Theta}\langle\kappa(\mathbf{f})\tau_{s}\tau_{d(J)}\rangle_{g-g^{'}}^{\Theta}.\end{aligned}\end{equation}
\end{theorem}

\begin{theorem}\label{RecursionB} Let $\mathbf{b}\in\mathbb{N}^{\infty}, g\geq 1, n\geq 1$ and $d_j\geq 0$. Then
\begin{equation}\label{Recursion2}\begin{aligned}&(2d_1+1)!!\langle\kappa(\mathbf{b})\prod\limits_{i=1}\limits^{n}\tau_{d_i} \rangle^{\Theta}_g\\&=\sum\limits_{j=2}\limits^{n}\sum\limits_{\mathbf{L}+\mathbf{L^{'}}=\mathbf{b}}\mathbf{\alpha_L}\binom{\mathbf{b}}{\mathbf{L}}\frac{(2b+2d_1+2d_j+1)!!}{(2d_j-1)!!}\langle\kappa(\mathbf{L^{'}})\tau_{|\mathbf{L}|+d_1+d_j}\prod\limits_{i\neq 1,j}\tau_{d_i} \rangle^{\Theta}_g\\&+\frac{1}{2}\sum\limits_{\mathbf{L}+\mathbf{L^{'}}=\mathbf{b}}\sum\limits_{r+s=|\mathbf{L}|+d_1-1}\alpha_{\mathbf{L}}\binom{\mathbf{b}}{\mathbf{L}}(2r+1)!!(2s+1)!!\langle\kappa(\mathbf{L^{'}})\tau_r\tau_s\prod\limits_{i=2}\limits^{n}\tau_{d_i} \rangle^{\Theta}_{g-1}\\&+\frac{1}{2}\sum\limits_{\mathbf{L}+\mathbf{e}+\mathbf{f}=\mathbf{b}}\sum_{\substack{r+s=|\mathbf{L}|+d_1-1\\[3pt]I\sqcup J=\{2,...,n\}}}\mathbf{\alpha_L}\binom{\mathbf{b}}{\mathbf{L},\mathbf{e},\mathbf{f}}(2r+1)!!(2s+1)!!\\&\qquad\qquad\qquad\qquad\qquad\qquad\qquad\qquad\qquad\times\langle\kappa(\mathbf{e})\tau_{r}\tau_{d(I)} \rangle^{\Theta}_{g^{'}}\langle\kappa(\mathbf{f})\tau_{s}\tau_{d(J)} \rangle^{\Theta}_{g-g^{'}},\end{aligned}\end{equation}where $\mathbf{\alpha_L}$ are determined by the following identity $$\sum\limits_{\mathbf{L}+\mathbf{L^{'}}=\mathbf{b}}\frac{(-1)^{||\mathbf{L}||}\mathbf{\alpha_L}}{\mathbf{L}!\mathbf{L^{'}}!(2|\mathbf{L^{'}}|-1)!!}=0,\qquad \mathbf{b}\neq \mathbf{0},$$ namely $$\mathbf{\alpha_{b}=\mathbf{b}!\sum_{\substack{\mathbf{L}+\mathbf{L^{'}}=\mathbf{b}\\[1pt]\mathbf{L^{'}}\neq \mathbf{0}}}}\frac{(-1)^{||\mathbf{L^{'}}||-1}\mathbf{\alpha_L}}{\mathbf{L}!\mathbf{L^{'}!}(2|\mathbf{L^{'}}|-1)!!},\qquad \mathbf{b}\neq \mathbf{0},$$ with initial data $\mathbf{\alpha}_\mathbf{0}=1.$
\end{theorem}

Alexandrov \cite{alexandrov2023cut} proved that the generating
function of super higher Weil-Petersson volumes was governed by
cut-and-join operators, which also gives an effective way to compute
these intersection numbers.

Norbury \cite{norbury2020enumerative} proved a super version of
Kauffman-Manin-Zagier's formula \cite{kaufmann1996higher} converting
mixed $\kappa$ and $\psi$ class integrals to pure $\psi$ class
integrals. We will prove Theorem \ref{RecursionA} from Norbury's
formula and \eqref{Recursionpsi}.

On the other hand, Theorem \ref{RecursionB} is obtained from Theorem
\ref{RecursionA} by using a trick of reciprocal multivariate series.
See Section 3 for details.

Denote $\alpha (l,0,...,0)$ by $\alpha_l$, we recover
\eqref{Recursionkappsi} from \eqref{Recursion2} with
$\alpha_l=l!\beta_l,$ where $\beta_l$ is defined by
$$\frac{1}{cos(\sqrt{2}x)}=\sum\limits_{l=0}\limits^{\infty}\beta_lx^{2l}.$$
We also have $$\alpha(\bm{\delta}_l)=\frac{1}{(2l+1)!!}.$$ Setting
$\mathbf{b}=\mathbf{0},$ we get \eqref{Recursionpsi}. So we may
apply Theorem \ref{RecursionB} to compute any super intersection
numbers of $\psi$ and $\kappa$ classes recursively with two initial
values,
$$\langle\tau_0\rangle_{1}^{\Theta}=\frac{1}{8},\qquad \langle\kappa_1\rangle_{2}^{\Theta}=\frac{3}{128}.$$

Note that Theorem \ref{RecursionA} and \ref{RecursionB} hold only
for $n\geq 1$. If $n=0$ i.e. for higher Weil-Petersson volumes
$V_g(\kappa(\mathbf{b}))$, we may apply the following formula first
(see Proposition 4.2 for a proof).
$$\langle\kappa(\mathbf{b}) \rangle_{g}^{\Theta}=\frac{1}{2g-2}\sum\limits_{\mathbf{L}+\mathbf{L^{'}}=\mathbf{b}}(-1)^{||\mathbf{L}||}\binom{\mathbf{b}}{\mathbf{L}}\langle\tau_{|\mathbf{L}|}\kappa(\mathbf{L}^{'})\rangle_{g}^{\Theta}.$$
For the particular case $\mathbf{b}=(m,0,0,...)$, it has been proved
by Norbury \cite{norbury2020enumerative}.

Our work parallels that of \cite{mulase2006mirzakhani,
liu2009recursion}. Moreover, we derive the following two recursion
formulae for super higher Weil-Petersson volumes (without $\psi$
classes).

\begin{theorem}\label{superHWP1} Let $g\geq 1$, $n\geq 0$ and $\mathbf{b}\in\mathbb{N}^{\infty}$. Then
\begin{equation}\label{HNSWP}
V_{g,n+1}(\kappa(\mathbf{b}))=(2g-2+n+||\mathbf{b}||)V_{g,n}(\kappa(\mathbf{b}))+\sum_{\substack{\mathbf{L}+\mathbf{L^{'}}=\mathbf{b}\\[1pt]||\mathbf{L^{'}}||\geq 2}}\binom{\mathbf{b}}{\mathbf{L}}V_{g,n}(\kappa(\mathbf{L})\kappa_{|\mathbf{L^{'}}|}).
\end{equation}
\end{theorem}
The above formula reduces the calculations of
$V_{g,n}(\kappa(\mathbf{b}))$ to $V_{g}(\kappa(\mathbf{b}))$.

\begin{theorem}\label{superHWP2} Let $g\geq 1$, $\mathbf{b}\in\mathbb{N}^{\infty}$. Then
\begin{equation}\label{HSWP}\begin{aligned}
||\mathbf{b}||V_{g}(\kappa(\mathbf{b}))=&\sum_{\substack{\mathbf{L}+\mathbf{L_1}+\mathbf{L_2}=\mathbf{b}\\[1pt]||\mathbf{L}||\geq 1}}(-1)^{||\mathbf{L}||-1}\binom{\mathbf{b}}{\mathbf{L},\mathbf{L_1},\mathbf{L_2}}V_{g}(\kappa(\mathbf{L_1})\kappa_{|\mathbf{L}|+|\mathbf{L_2}|})\\&\qquad-\sum_{\substack{\mathbf{L}+\mathbf{L^{'}}=\mathbf{b}\\[1pt]||\mathbf{L^{'}}||\geq 2}}V_{g}(\kappa(\mathbf{L})\kappa_{|\mathbf{L^{'}}|}).\end{aligned}
\end{equation}
\end{theorem}
However, \eqref{HSWP} is not an effective formula to calculate
$V_{g}(\kappa(\mathbf{b}))$, although it does give nontrivial
relations among $V_{g}(\kappa(\mathbf{b}))$.

Define the generating function
$G^{\Theta}(\hbar,\mathbf{s},\mathbf{t})$ of
$\langle\kappa(\mathbf{b})\prod\limits_{i=1}\limits^{n}\tau_{d_i}\rangle_{g}^{\Theta}$
as
\begin{equation}\label{GTheta}G^{\Theta}(\hbar,\mathbf{s},\mathbf{t}):=\mathrm{exp}\
\sum\limits_{g,n}\frac{\hbar^{g-1}}{n!}\sum\limits_{\mathbf{d}\in
N^{n}}\int_{\overline{\mathcal{M}}_{g,n}}\Theta_{g,n}\prod\limits_{i=1}\limits^{n}\psi_i^{d_i}t_{d_i}\prod\limits_{j=1}\limits^{\infty}\kappa_j^{m_j}\frac{s_j^{m_j}}{m_j!}.\end{equation}

From Theorem \ref{RecursionA} and \ref{RecursionB}, we gave a new
proof of a theorem of Norbury that that by a suitable parameter
shift, the generating function
$G^{\Theta}(\hbar,\mathbf{s},\mathbf{t})$ is still annihalated by
Virasoro constraints.
\begin{theorem}\textup{\cite[Theorem 5.7]{norbury2020enumerative}}\label{Vir}
We have $$G^{\Theta}(\hbar,\mathbf{s},t_0,t_1,...)=Z^{\Theta}(t_0,t_1+p_1(\mathbf{s}),t_2+p_2(\mathbf{s}),...),$$ where $p_k$ are polynomials in $\mathbf{s}$ given by $$p_k(\mathbf{s})=\sum\limits_{||\mathbf{L}||=k}\frac{(-1)^{||\mathbf{L}||-1}}{\mathbf{L!}}\mathbf{s}^\mathbf{L},$$ where $\mathbf{s}^\mathbf{L}=\prod\limits_{i\geq 1}s_i^{L(i)}.$
In particular, for any $\mathbf{s}\in N^{\infty}$, $G^{\Theta}(\hbar,\mathbf{s},\mathbf{t})$ is a tau function of KdV hierarchy.
\end{theorem}

The paper is organized as follows. In Section 2, we give a brief overview of KdV hierarchy and
Virasoro constraints. In Section 3, we prove
Theorem \ref{RecursionA} and \ref{RecursionB}. In Section 4, we prove Theorem \ref{superHWP1} and Theorem \ref{superHWP2}. In Section 5, we give a proof of Theorem \ref{Vir} from Theorem
\ref{RecursionA} and \ref{RecursionB}. In the Appendix, we rewrite
Stanford-Witten's recursion formula as a recursion formula of super intersection numbers.

\

\noindent{\bf Acknowledgements} We thank Di Yang for very helpful
comments.

\vskip 20pt
\section{KdV hierarchy and Virasoro constraints}
\setcounter{equation}{0}

In this section we record some recursion formulae of super intersection numbers involving $\Theta$ and $\psi$ classes from KdV hierarchy and Virasoro constraints.

\subsection{KdV hierarchy} A fuction $Z(\hbar,t_0,t_1,...)$ is said to be the tau function of KdV hierarchy if $U=\hbar\frac{\partial^2\mathrm{log}Z}{\partial t_0^2}$ is a solution of the KdV hierarchy
\begin{equation}\label{KdV}U_{t_n}=\frac{\partial}{\partial t_0}R_{n+1}(\hbar,U,U_{t_0},U_{t_0t_0},...)\ \  \text{with}\ \ U(\hbar,t_0,t_1,...)|_{t_{i>0}=0}=f(\hbar,t_0),\end{equation} where $R_{n+1}(\hbar,U,U_{t_0},U_{t_0t_0},...)$ are certain polynomials in $U$ and its $t_0$ derivatives, and can be defined inductively by the formulae
$$R_1=U,\qquad \frac{\partial R_{n+1}}{t_0}=\frac{1}{2n+1}(R_n\frac{\partial U}{t_0}+2U\frac{\partial R_n}{\partial t_0}+\frac{\hbar}{4}\frac{\partial^3 R_n}{\partial t_0^3}).$$
The first equation in this hierarchy is the KdV equation \begin{equation}\label{KdVequation}U_{t_1}=UU_{t_0}+\frac{\hbar}{12}U_{t_0t_0t_0}.\end{equation}

The Witten-Kontsevich theorem is related to matrix Airy model \cite{kontsevich1992intersection,witten1990two}
$$f(\hbar,t_0)=\hbar t_0.$$
On the other hand, the Br{\'e}zin-Gross-Witten solution appears in the unitary matrix model \cite{brezin1980external,gross1980possible} with
$$f(\hbar,t_0)=\frac{\hbar}{8(1-t_0)^2}.$$

\subsection{Virasoro constraints} It is also known that \eqref{Witten-Kontsevich} and \eqref{ZTheta} both obey the Virasoro constraints \cite{dijkgraaf1991loop,gross1992unitary,kontsevich1992intersection,witten1990two}.

For each integer $m\geq -1$, define the following differential operators
\begin{align*}\hat{\mathcal{L}}_m=&\frac{\hbar}{2}\sum\limits_{i+j=m-1}(2i+1)!!(2j+1)!!\frac{\partial^2}{\partial t_i\partial t_j}+\sum\limits_{i=0}\limits^{\infty}\frac{(2i+2m+1)!!}{(2i-1)!!}t_i\frac{\partial}{\partial t_{i+m}}\\&+\frac{1}{8}\delta_{m,0}+\frac{t_0^2}{2\hbar}\delta_{m,-1}.\end{align*}
Then we have\begin{equation}\label{VirasoroKW}\mathcal{L}_k Z^{KW}(\hbar,t_0,t_1,...)=0,\qquad k=-1,0,1,...\end{equation}
\begin{equation}\label{VirasoroBGW}\mathcal{L}^{'}_k Z^{BGW}(\hbar,t_0,t_1,...)=0,\qquad k=0,1,2,...\end{equation}
with $$\mathcal{L}_k=-\frac{1}{2}\left((2m+3)!!\frac{\partial}{\partial t_{k+1}}-\hat{\mathcal{L}}_k\right),\qquad \mathcal{L}^{'}_k=-\frac{1}{2}\left((2m+1)!!\frac{\partial}{\partial t_{k+1}}-\hat{\mathcal{L}}_k\right).$$

\subsection{Intersection numbers} The KdV equation \eqref{KdV} and Virasoro constraints \eqref{VirasoroKW} give rise to the following formulae respectively.
$$\langle\tau_{0}\tau_1\prod\limits_{i=1}\limits^{n}\tau_{d_i}\rangle_{g}=\frac{1}{2}\sum\limits_{I\sqcup J=\{1,...n\}}\langle\tau_{0}^2\tau_{d(I)}\rangle_{g^{'}}\langle\tau_{0}^2\tau_{d(J)}\rangle_{g-g^{'}}+\frac{1}{12}\langle\tau_0^{4}\prod\limits_{i=1}\limits^{n}\tau_{d_i}\rangle_{g-1},$$
and
\begin{align*}\langle\tau_{k+1}\prod\limits_{i=1}\limits^{n}\tau_{d_i} \rangle_g=\frac{1}{(2k+3)!!}[&\sum\limits_{i=1}\limits^{n}\frac{(2k+2d_i+1)!!}{(2d_i-1)!!}\langle\tau_{k+d_i}\prod\limits_{j\neq i}\tau_{d_j} \rangle_g\\&+\frac{1}{2}\underset{r+s=k-1}{\sum}(2r+1)!!(2s+1)!!\langle\tau_r\tau_s\prod\limits_{i=1}\limits^{n}\tau_{d_i} \rangle_{g-1}\\+\frac{1}{2}\underset{r+s=k-1}{\sum}&(2r+1)!!(2s+1)!!\sum_{\substack{g_1+g_2=g\\[3pt]I\sqcup J=\{1,...,n\}}}\langle\tau_{r}\tau_{d(I)} \rangle_{g_1}\langle\tau_{s}\tau_{d(J)} \rangle_{g_2}],\end{align*}
which is known as DVV formula \cite{dijkgraaf1991loop}.

Similarly, the following formula is derived from KdV equation \eqref{KdV}.
\begin{equation}\label{KdVrecursion}\langle\tau_{0}\tau_1\prod\limits_{i=1}\limits^{n}\tau_{d_i}\rangle_{g}^{\Theta}=\frac{1}{2}\sum\limits_{I\sqcup J=\{1,...n\}}\langle\tau_{0}^2\tau_{d(I)}\rangle_{g^{'}}^{\Theta}\langle\tau_{0}^2\tau_{d(J)}\rangle_{g-g^{'}}^{\Theta}+\frac{1}{12}\langle\tau_0^{4}\prod\limits_{i=1}\limits^{n}\tau_{d_i}\rangle_{g-1}^{\Theta}\end{equation}
The Virasoro constraints \eqref{VirasoroBGW} implies \eqref{Recursionpsi}. In particular, setting $k=0$ in \eqref{Recursionpsi}, we get the dilaton equation
\begin{equation}\label{dilaton}
\langle\tau_{0}\prod\limits_{i=1}\limits^{n}\tau_{d_i}
\rangle^{\Theta}_g
=(2g-2+n)\langle\prod\limits_{i=1}\limits^{n}\tau_{d_i}
\rangle^{\Theta}_g.\end{equation}

We have the following two closed formulae for intersection numbers.
The first one is due to Norbury \cite[Appendix]{norbury2023new}. The
second one is due to Bertola-Ruzza \cite[Corollary
1.2]{bertola2019brezin} and Dubrovin-Yang-Zagier \cite[Corollary
3]{dubrovin2021tau} with different proofs.

\begin{proposition}
\textup{(1)}
$\displaystyle\langle\tau_{0}^n\rangle^{\Theta}_1=\frac{(n-1)!}{8};$

\textup{(2)}
$\displaystyle\langle\tau_{g-1}\rangle^{\Theta}_g=\frac{(2g-1)!!^2}{8^gg!(2g-1)},\qquad\forall
g\geq 1.$
\end{proposition}
\begin{proof}
(1) is a simple corollary of \eqref{dilaton} and
$\langle\tau_{0}\rangle^{\Theta}_1=\frac{1}{8}$.

Now we prove (2) by induction on $g$. From \eqref{KdVrecursion} and
\eqref{dilaton}, we obtain the following
\begin{align*}\langle\tau_0\tau_1\tau_{g-1}\rangle^{\Theta}_g&=\langle\tau_0^2\tau_{g-2}\rangle^{\Theta}_{g-1}\langle\tau_0^2\rangle^{\Theta}_1+\frac{1}{12}\langle\tau_0^4\tau_{g-2}\rangle^{\Theta}_{g-1}\\&=\bigg(\frac{(2g-2)(2g-3)}{8}+\frac{2g(2g-1)(2g-2)(2g-3)}{12}\bigg)\langle\tau_{g-2}\rangle^{\Theta}_{g-1}.\end{align*}
From \eqref{Recursionpsi} and \eqref{dilaton}, we have
\begin{align*}\langle\tau_0\tau_1\tau_{g-1}\rangle^{\Theta}_g&=2g\langle\tau_1\tau_{g-1}\rangle^{\Theta}_g\\&=\frac{2g}{3}\times\bigg[(2g-1)(2g-3)\langle\tau_{g-1}\rangle^{\Theta}_g\\&\qquad\qquad+\bigg(\frac{(2g-2)(2g-3)}{2}+\frac{2g-3}{8}\bigg)\langle\tau_{g-2}\rangle^{\Theta}_{g-1}\bigg].\end{align*}
Therefore, $$\langle\tau_{g-1}\rangle^{\Theta}_g=\frac{(2g-1)(2g-3)}{8g}\langle\tau_{g-2}\rangle^{\Theta}_{g-1}.$$
Note that $\displaystyle\langle\tau_{0}\rangle^{\Theta}_1=\frac{1}{8}.$ Hence, (2) follows.
\end{proof}

In \cite{Huang}, the first author obtained a 2-point closed formula.
\begin{proposition}\textup{\cite{Huang}}
For integers $g\geq 1$ and $k\geq 0$ such that $2k\leq g-1$, one has
\begin{align*}\langle\tau_{k}\tau_{g-1-k} \rangle_{g}^\Theta=\frac{(2g-1)!!^2}{8^gg!}\cdot\frac{(2g-1)!!}{(2k+1)!!(2g-1-2k)!!}\cdot\sum\limits_{i=0}\limits^{k}\frac{g-2i}{g}\binom{g}{i}^4\binom{2g}{2i}^{-3}.\end{align*}
\end{proposition}
For example, $$\langle\tau_{1}\tau_{1} \rangle_{3}^\Theta=\frac{63}{512},\quad \langle\tau_{2}\tau_{3} \rangle_{6}^\Theta=\frac{7949025}{2097152},\quad \langle\tau_{4}\tau_{4} \rangle_{9}^\Theta=\frac{8093029715505}{8589934592}.$$

\vskip 20pt
\section{Proof of Theorems \ref{RecursionA} and \ref{RecursionB}.}
\setcounter{equation}{0}
We need the following elementary lemma.
\begin{lemma}\textup{\cite[Lemma 2.1]{liu2009recursion}}
Let $F(\mathbf{L},n)$ and $G(\mathbf{L},n)$ be two functions defined on $N^\infty\times\mathbb{N},$ where $\mathbb{N}$ is the set of nonnegative integers. Let $\alpha_{\mathbf{L}}$ and $\beta_{\mathbf{L}}$ be real numbers depending only on $\mathbf{L}\in N^\infty$ that satisfy $\alpha_{\mathbf{0}}\beta_{\mathbf{0}}=1$ and $$\sum\limits_{\mathbf{L}+\mathbf{L^{'}}=\mathbf{b}}\alpha_{\mathbf{L}}\beta_{\mathbf{L^{'}}}=0,\qquad \mathbf{b}\neq\mathbf{0}.$$ Then the following two identities are equivalent:
$$\setlength{\abovedisplayskip}{3pt}
\setlength{\belowdisplayskip}{3pt}G(m,n)=\sum\limits_{\mathbf{L}+\mathbf{L^{'}}=\mathbf{b}}\alpha_\mathbf{L} F(\mathbf{L^{'}},n+|\mathbf{L}|),\qquad\forall (\mathbf{b},n)\in N^\infty\times\mathbb{N},$$
$$\setlength{\abovedisplayskip}{3pt}
\setlength{\belowdisplayskip}{3pt}F(m,n)=\sum\limits_{\mathbf{L}+\mathbf{L^{'}}=\mathbf{b}}\alpha_\mathbf{L} G(\mathbf{L^{'}},n+|\mathbf{L}|),\qquad\forall (\mathbf{b},n)\in N^\infty\times\mathbb{N}.$$
\end{lemma}
\begin{proof}
Assume the first identity holds, then we have
\begin{align*}\sum\limits_{\mathbf{a}=\mathbf{0}}\limits^{\mathbf{b}}\beta_\mathbf{a} G(\mathbf{b}-\mathbf{a},n+|\mathbf{a}|)&=\sum\limits_{\mathbf{a}=\mathbf{0}}\limits^{\mathbf{b}}\beta_\mathbf{a}\sum\limits_{\mathbf{a^{'}}=\mathbf{0}}\limits^{\mathbf{b}-\mathbf{a}}\alpha_\mathbf{a^{'}} F(\mathbf{b}-\mathbf{a}-\mathbf{a^{'}},n+|\mathbf{a}+\mathbf{a^{'}}|)\\&=\sum\limits_{\mathbf{L}=\mathbf{0}}\limits^{\mathbf{b}}\ \sum\limits_{\mathbf{a}+\mathbf{a^{'}}=\mathbf{b}}(\beta_\mathbf{a}\alpha_\mathbf{a^{'}})F(\mathbf{b}-\mathbf{L},n+|\mathbf{L}|)\\&=\sum\limits_{\mathbf{L}=\mathbf{0}}\limits^{\mathbf{b}}\bm{\delta}_{\mathbf{L},\mathbf{0}}F(\mathbf{b}-\mathbf{L},n+|\mathbf{L}|)\\&=F(\mathbf{b},n).
\end{align*}
Hence, the second identity is derived. The proof of the other direction is the same.
\end{proof}

It is well-known that the mixture of $\kappa$ classes and $\psi$ classes can be reduced to pure $\psi$ classes. Kaufman-Manin-Zagier \cite{kaufmann1996higher} derived the following relation via projection formula
$$\kappa(\mathbf{b})\prod\limits_{i=1}\limits^{n}\tau_{d_i}=\pi_\ast\bigg(\sum\limits_{k=1}\limits^{||\mathbf{b}||}\frac{(-1)^{||\mathbf{b}||-k}}{k!}\sum_{\substack{\mathbf{m_1}+\cdots+\mathbf{m_k}=\mathbf{b}\\[3pt]\mathbf{m_i}\neq\mathbf{0}}}\binom{\mathbf{b}}{\mathbf{m_1},...,\mathbf{m_k}}\prod\limits_{i=1}\limits^{n}\tau_{d_i}\prod\limits_{j=1}\limits^{k}\tau_{|\mathbf{m_j}|+1}\bigg),$$
where $\pi: \overline{\mathcal{M}}_{g,n+k}\to\overline{\mathcal{M}}_{g,n}$ is the morphism which forgets the last $k$ marked points. If we add $\Theta$ class into both sides above, the analogous relation also holds. See \cite[Page 58]{norbury2020enumerative}.
\begin{lemma}\textup{\cite{norbury2020enumerative}}
$$\Theta_{g,n}\kappa(\mathbf{b})\prod\limits_{i=1}\limits^{n}\tau_{d_i}=\pi_\ast\bigg(\sum\limits_{k=1}\limits^{||\mathbf{b}||}\frac{(-1)^{||\mathbf{b}||-k}}{k!}\Theta_{g,n+k}\sum_{\substack{\mathbf{m_1}+\cdots+\mathbf{m_k}=\mathbf{b}\\[3pt]\mathbf{m_i}\neq\mathbf{0}}}\binom{\mathbf{b}}{\mathbf{m_1},...,\mathbf{m_k}}\prod\limits_{i=1}\limits^{n}\tau_{d_i}\prod\limits_{j=1}\limits^{k}\tau_{|\mathbf{m_j}|}\bigg).$$
\end{lemma}\noindent\textbf{Proof of Theorem \ref{RecursionA}.}
From Lemma 3.2 and \eqref{Recursionpsi}, we have
\begin{align*}&(2d_1+1)!!\langle\prod\limits_{d=1}\limits^{n}\tau_{d_i}\kappa(\mathbf{b})\rangle_{g}^{\Theta}\\&=(2d_1+1)!!\sum\limits_{k=1}\limits^{||\mathbf{b}||}\frac{(-1)^{||\mathbf{b}||-k}}{k!}\sum_{\substack{\mathbf{m_1}+\cdots+\mathbf{m_k}=\mathbf{b}\\[3pt]\mathbf{m_i}\neq\mathbf{0}}}\binom{\mathbf{b}}{\mathbf{m_1},...,\mathbf{m_k}}\langle\prod\limits_{i=1}\limits^{n}\tau_{d_i}\prod\limits_{j=1}\limits^{k}\tau_{|\mathbf{m_j}|}\rangle_{g}^{\Theta}\\&=\sum\limits_{k=1}\limits^{||\mathbf{b}||}\frac{(-1)^{||\mathbf{b}||-k}}{k!}\sum_{\substack{\mathbf{m_1}+\cdots+\mathbf{m_k}=\mathbf{b}\\[3pt]\mathbf{m_i}\neq\mathbf{0}}}\binom{\mathbf{b}}{\mathbf{m_1},...,\mathbf{m_k}}\\&\qquad\qquad\times\bigg(\sum\limits_{j=2}\limits^{n}\frac{(2d_1+2d_j+1)!!}{(2d_j-1)!!}\langle\tau_{d_1+d_j}\prod\limits_{i\neq 1,j}\tau_{d_j}\prod\limits_{i=1}\limits^{k}\tau_{|\mathbf{m_i}|} \rangle^{\Theta}_g\\&\qquad\qquad\qquad+\sum\limits_{j=1}\limits^{k}\frac{(2d_1+2m_j+1)!!}{(2m_j-1)!!}\langle\tau_{d_1+m_j}\prod\limits_{i=2}\limits^{n}\tau_{d_j}\prod\limits_{i\neq j}\tau_{|\mathbf{m_i}|} \rangle^{\Theta}_g\\&\qquad\qquad\qquad+\frac{1}{2}\underset{r+s=k_1-1}{\sum}(2r+1)!!(2s+1)!!\langle\tau_r\tau_s\prod\limits_{i=2}\limits^{n}\tau_{d_i}\prod\limits_{i=1}\limits^{k}\tau_{|\mathbf{m_i}|} \rangle^{\Theta}_{g-1}\\&+\frac{1}{2}\underset{r+s=k_1-1}{\sum}(2r+1)!!(2s+1)!!\sum_{\substack{I\sqcup J=\{2,...,n\}\\[3pt]I^{'}\sqcup J^{'}=\{1,...,k\}}}\langle\tau_{r}\tau_{d(I)} \prod\limits_{i\in I^{'}}\tau_{|\mathbf{m_i}|}\rangle^{\Theta}_{g^{'}}\langle\tau_{s}\tau_{d(I^{'})}\prod\limits_{j\in J^{'}}\tau_{|\mathbf{m_j}|} \rangle^{\Theta}_{g-g^{'}}\bigg)\\&=\sum\limits_{j=2}\limits^{n}\frac{(2d_1+2d_j+1)!!}{(2d_j-1)!!}\langle\kappa(\mathbf{b})\prod_{i\neq 1,j}\limits^{n}\tau_{d_i}\rangle_{g}^{\Theta}\\&\qquad+\frac{1}{2}\sum\limits_{r+s=d_1-1}(2r+1)!!(2s+1)!!\langle\kappa(\mathbf{b})\tau_{r}\tau_s\prod\limits_{i=2}\limits^{n}\tau_{d_i}\rangle_{g-1}^{\Theta}\\&+\frac{1}{2}\sum_{\substack{\mathbf{e}+\mathbf{f}=b\\[3pt]I\sqcup J=\{2,...,n\}}}\sum\limits_{r+s=d_1-1}\binom{\mathbf{b}}{\mathbf{e}}(2r+1)!!(2s+1)!!\langle\kappa(\mathbf{e})\tau_{r}\tau_{d(I)}\rangle_{g^{'}}^{\Theta}\langle\kappa(\mathbf{f})\tau_{s}\tau_{d(J)}\rangle_{g-g^{'}}^{\Theta}\\&+\sum\limits_{k=1}\limits^{||\mathbf{b}||}\frac{(-1)^{||\mathbf{b}||-k}}{k!}\sum_{\substack{\mathbf{m_1}+\cdots+\mathbf{m_k}=\mathbf{b}\\[3pt]\mathbf{m_i}\neq\mathbf{0}}}\binom{\mathbf{b}}{\mathbf{m_1},...,\mathbf{m_k}}\\&\qquad\qquad\qquad\qquad\qquad\times\sum\limits_{j=1}\limits^{k}\frac{(2d_1+2m_j+1)!!}{(2m_j-1)!!}\langle\tau_{d_1+m_j}\prod\limits_{i=2}\limits^{n}\tau_{d_j}\prod\limits_{i\neq j}\tau_{|\mathbf{m_i}|} \rangle^{\Theta}_g  \\&=RHS+\sum\limits_{k=0}\limits^{||\mathbf{b}||-1}\frac{(-1)^{||\mathbf{b}||-k-1}}{(k+1)!}\sum_{\substack{\mathbf{L}+\mathbf{L^{'}}=\mathbf{b}\\[3pt]\mathbf{L}\neq\mathbf{0}}}\sum_{\substack{\mathbf{m_1}+\cdots+\mathbf{m_k}=\mathbf{L^{'}}\\[3pt]\mathbf{m_i}\neq\mathbf{0}}}\binom{\mathbf{b}}{\mathbf{L}}\binom{\mathbf{L^{'}}}{\mathbf{m_1},...,\mathbf{m_k}}\\&\qquad\qquad\qquad\qquad\qquad\qquad\times(k+1)\frac{(2d_1+2|\mathbf{L}|+1)!!}{(2|\mathbf{L}|-1)!!}\langle\tau_{d_1+|\mathbf{L}|}\prod\limits_{j=2}\limits^{n}\tau_{d_j}\prod\limits_{i=1}\limits^{n}\tau_{|\mathbf{m_i}|} \rangle^{\Theta}_g  \\&=RHS-\sum_{\substack{\mathbf{L}+\mathbf{L^{'}}=\mathbf{b}\\[3pt]\mathbf{L}\neq\mathbf{0}}}(-1)^{||\mathbf{L}||}\binom{\mathbf{b}}{\mathbf{L}}\frac{(2d_1+2|\mathbf{L}|+1)!!}{(2|\mathbf{L}|-1)!!}\langle\kappa(\mathbf{L^{'}})\tau_{d_1+|\mathbf{L}|}\prod\limits_{j=2}\limits^{n}\tau_{d_j}\rangle_{g}^{\Theta}\\&=RHS-LHS+(2d_1+1)!!\langle\prod\limits_{d=1}\limits^{n}\tau_{d_i}\kappa(\mathbf{b})\rangle_{g}^{\Theta}\end{align*}
So we have RHS=LHS. \qed

We will prove Theorem \ref{RecursionB} from Theorem \ref{RecursionA} and Lemma 3.1.\\

\noindent\textbf{Proof of Theorem \ref{RecursionB}.}

Denote $$F(\mathbf{b},d_1)=\frac{(2d_1+1)!!}{\mathbf{b!}}\langle\kappa(\mathbf{b})\prod\limits_{i=1}\limits^{n}\tau_{d_i} \rangle^{\Theta}_g,$$
and
\begin{align*}G(\mathbf{b},k_1)&=\sum\limits_{j=2}\limits^{n}\frac{(2|\mathbf{b}|+2d_1+2d_j+1)!!}{\mathbf{b!}(2d_j-1)!!}\langle\kappa(\mathbf{b})\tau_{b+d_1+d_j}\prod_{i\neq 1,j}\tau_{d_i}\rangle_{g}^{\Theta}\\&\qquad+\frac{1}{2}\sum\limits_{r+s=d_1-1}\frac{(2r+1)!!(2s+1)!!}{\mathbf{b!}}\langle\kappa(\mathbf{b})\tau_{r}\tau_s\prod\limits_{i=2}\limits^{n}\tau_{d_i}\rangle_{g-1}^{\Theta}\\&+\frac{1}{2}\sum_{\substack{\mathbf{e}+\mathbf{f}=\mathbf{b}\\[3pt]I\sqcup J=\{2,...,n\}}}\sum\limits_{r+s=d_1-1}\frac{(2r+1)!!(2s+1)!!}{\mathbf{e!}\mathbf{f!}}\langle\kappa(\mathbf{e})\tau_{r}\tau_{d(I)}\rangle_{g^{'}}^{\Theta}\langle\kappa(\mathbf{f})\tau_{s}\tau_{d(J)}\rangle_{g-g^{'}}^{\Theta}.\end{align*}
By Theorem \ref{RecursionA}, we have $$\sum\limits_{\mathbf{L}+\mathbf{L^{'}}=\mathbf{b}}\frac{(-1)^{||\mathbf{L}||}}{\mathbf{L!}(2|\mathbf{L}|-1)!!}F(\mathbf{L^{'}},d_1+|\mathbf{L}|)=G(\mathbf{b},d_1),$$
then by Lemma 3.1
$$F(\mathbf{b},d_1)=\sum\limits_{\mathbf{L}+\mathbf{L^{'}}=\mathbf{b}}\frac{\alpha_{\mathbf{L}}}{\mathbf{L!}}G(\mathbf{L^{'}},d_1+|\mathbf{L}|).$$
Therefore, Theorem \ref{RecursionB} is proved. \qed

\vskip 20pt
\section{super Higher Weil-Petersson volumes}
By applying Lemma 3.2 as in the proof of Theorem \ref{RecursionA},
we may generalize recursions of $\Theta$ and pure $\psi$ classes to
recursions involving $\Theta$, $\psi$ and $\kappa$ classes. Here we
obtain the generalization of \eqref{KdVrecursion}.

\begin{proposition} For $\mathbf{b}\in N^\infty$, $g\geq 1$, $n\geq 0$ and $d_j\geq 0$,we have
\begin{align*}\langle\tau_{0}\tau_1\prod\limits_{i=1}\limits^{n}\tau_{d_i}\kappa(\mathbf{b})\rangle_{g}^{\Theta}=&\frac{1}{2}\sum_{\substack{\mathbf{e}+\mathbf{f}=\mathbf{b}\\[3pt]I\sqcup J=\underline{n}}}\binom{\mathbf{b}}{\mathbf{e}}\langle\tau_{0}^2\tau_{d(I)}\kappa(\mathbf{e})\rangle_{g^{'}}^{\Theta}\langle\tau_0^2\tau_{d(J)}\kappa(\mathbf{f})\rangle_{g-g^{'}}^{\Theta}\\&\qquad\qquad\qquad+\frac{1}{12}\langle\tau_0^{4}\prod\limits_{i=1}\limits^{n}\tau_{d_i}\kappa(\mathbf{b})\rangle_{g-1}^{\Theta},\end{align*}
where $\underline{n}=\{1,...,n\}$.
\end{proposition}

Moreover, the following is a generalization of the dilaton equation \eqref{dilaton}.

\begin{proposition} For $\mathbf{b}\in N^\infty$, $n\geq 0$, $d_j\geq 0$ and $g\geq 1$,
$$\sum\limits_{\mathbf{L}+\mathbf{L^{'}}=\mathbf{b}}(-1)^{||\mathbf{L}||}\binom{\mathbf{b}}{\mathbf{L}}\langle\tau_{|\mathbf{L}|}\prod\limits_{i=1}\limits^{n}\tau_{d_i}\kappa(\mathbf{L^{'}}) \rangle_{g}^{\Theta}=(2g-2+n)\langle\prod\limits_{i=1}\limits^{n}\tau_{d_i}\kappa(\mathbf{b}) \rangle_{g}^{\Theta}$$
\end{proposition}
\begin{proof}
By dilaton equation \eqref{dilaton},
\begin{align*}
\langle\tau_{0}\prod\limits_{i=1}\limits^{n}\tau_{d_i}\kappa(\mathbf{b}) \rangle_{g}^{\Theta}&=\sum\limits_{k=1}\limits^{||\mathbf{b}||}\frac{(-1)^{||\mathbf{b}||-k}}{k!}\sum_{\substack{\mathbf{m_1}+\cdots+\mathbf{m_k}=\mathbf{b}\\[3pt]\mathbf{m_i}\neq\mathbf{0}}}\binom{\mathbf{b}}{\mathbf{m_1},...,\mathbf{m_k}}\langle\tau_0\prod\limits_{i=1}\limits^{n}\tau_{d_i}\prod\limits_{j=1}\limits^{k}\tau_{|\mathbf{m_j}|}\rangle_{g}^{\Theta}\\&=(2g-2+n)\langle\prod\limits_{i=1}\limits^{n}\tau_{d_i}\kappa(\mathbf{b}) \rangle_{g}^{\Theta}+\sum\limits_{k=0}\limits^{||\mathbf{b}||-1}\frac{(-1)^{||\mathbf{b}||-k-1}}{k!}\\&\qquad\qquad\times\sum_{\substack{\mathbf{L}+\mathbf{m_1}+\cdots+\mathbf{m_n}=\mathbf{b}\\[3pt]\mathbf{L}\neq\mathbf{0},\mathbf{m_i}\neq\mathbf{0}}}\binom{\mathbf{b}}{\mathbf{L},\mathbf{m_1},...,\mathbf{m_k}}\langle\tau_{|\mathbf{L}|}\prod\limits_{j=1}\limits^{n}\tau_{d_j}\prod\limits_{j=1}\limits^{k}\tau_{|\mathbf{m_j}|} \rangle^{\Theta}_g
\end{align*}
It follows by subtracting the last term from each side.
\end{proof}

We then need the following property from \cite{arbarello1996combinatorial,norbury2023new} to prove Theorem \ref{superHWP1}.
\begin{lemma}Let $\pi_{n+1}:\overline{M}_{g,n+1}\to\overline{M}_{g,n}$ be the morphism that forgets the last marked point.
\begin{enumerate}[(1)]
\item $\pi_{n\ast}(\Theta_{g,n}\psi_1^{a_1}\cdots\psi_{n-1}^{a_{n-1}}\psi_n^{a_n})=\Theta_{g,n-1}\prod\limits_{i=1}\limits^{i}\psi_i^{a_i}\kappa_{a_n},\qquad$ for $a_j\geq0$;
\item $\kappa_a=\pi_{n+1}^\ast(\kappa_a)+\psi_{n+1}^{a}$ on $\overline{M}_{g,n+1}$;
\item $\kappa_0=2g-2+n$ on $\overline{M}_{g,n}$.
\end{enumerate}
\end{lemma}
\noindent\textbf{Proof of Theorem \ref{superHWP1}.} From Lemma 4.3, we have
\begin{align*}
V_{g,n+1}(\kappa(\mathbf{b}))&=\int_{\overline{\mathcal{M}}_{g,n+1}}\Theta_{g,n+1}\kappa(\mathbf{b})\\&=\int_{\overline{\mathcal{M}}_{g,n+1}}\psi_{n+1}\pi_{n+1}^\ast\Theta_{g,n}\cdot\prod\limits_{i\geq 1}(\pi_{n+1}^\ast(\kappa_i)+\psi_{n+1}^i)^{b(i)}\\&=\sum\limits_{\mathbf{L}+\mathbf{L^{'}}=\mathbf{b}}\int_{\overline{\mathcal{M}}_{g,n}}\Theta_{g,n}\kappa(\mathbf{L})\kappa_{|\mathbf{L^{'}}|}\\&=(2g-2+n+||\mathbf{b}||)V_{g,n}(\kappa(\mathbf{b}))+\sum_{\substack{\mathbf{L}+\mathbf{L^{'}}=\mathbf{b}\\[1pt]||\mathbf{L^{'}}||\geq 2}}\binom{\mathbf{b}}{\mathbf{L}}V_{g,n}(\kappa(\mathbf{L})\kappa_{|\mathbf{L^{'}}|}).
\end{align*}
Hence, Theorem \ref{superHWP1} is proved.\qed

\noindent\textbf{Proof of Theorem \ref{superHWP2}.}
Let $n=0$ in Proposition 4.2 and by Theorem \ref{superHWP1}, we have
\begin{align*}\setlength{\abovedisplayskip}{3pt}
\setlength{\belowdisplayskip}{3pt}
(2g-2)V_{g}(\kappa(\mathbf{b}))&=\sum_{\substack{\mathbf{L}+\mathbf{L^{'}}=\mathbf{b}\\[1pt]||\mathbf{L}||\geq 1}}(-1)^{||\mathbf{L}||}\binom{\mathbf{b}}{\mathbf{L}}\langle\tau_{|\mathbf{L}|}\kappa(\mathbf{L^{'}}) \rangle_{g}^{\Theta}+V_{g,1}(\kappa(\mathbf{b}))\\&=\sum_{\substack{\mathbf{L}+\mathbf{L^{'}}=\mathbf{b}\\[1pt]||\mathbf{L}||\geq 1}}(-1)^{||\mathbf{L}||}\binom{\mathbf{b}}{\mathbf{L}}\int_{\overline{\mathcal{M}}_{g,1}}\Theta_{g,1}\tau_1^{|\mathbf{L}|}\kappa(\mathbf{L^{'}})\\&\qquad+(2g-2+||\mathbf{b}||)V_{g}(\kappa(\mathbf{b}))+\sum_{\substack{\mathbf{L}+\mathbf{L^{'}}=\mathbf{b}\\[1pt]||\mathbf{L^{'}}||\geq 2}}\binom{\mathbf{b}}{\mathbf{L}}V_{g}(\kappa(\mathbf{L})\kappa_{|\mathbf{L^{'}}|})\\&=\sum_{\substack{\mathbf{L}+\mathbf{L^{'}}=\mathbf{b}\\[1pt]||\mathbf{L}||\geq 1}}(-1)^{||\mathbf{L}||}\binom{\mathbf{b}}{\mathbf{L}}\int_{\overline{\mathcal{M}}_{g,1}}\tau_1^{|\mathbf{L}|+1}\pi_1^\ast\Theta_{g}\prod\limits_{i\geq 1}(\pi_1^\ast(\kappa_i)+\psi_1^i)^{L^{'}(i)}\\&\qquad+(2g-2+||\mathbf{b}||)V_{g}(\kappa(\mathbf{b}))+\sum_{\substack{\mathbf{L}+\mathbf{L^{'}}=\mathbf{b}\\[1pt]||\mathbf{L^{'}}||\geq 2}}\binom{\mathbf{b}}{\mathbf{L}}V_{g}(\kappa(\mathbf{L})\kappa_{|\mathbf{L^{'}}|})\\&=\sum_{\substack{\mathbf{L}+\mathbf{L_1}+\mathbf{L_2}=\mathbf{b}\\[1pt]||\mathbf{L}||\geq 1}}(-1)^{||\mathbf{L}||}\binom{\mathbf{b}}{\mathbf{L},\mathbf{L_1},\mathbf{L_2}}V_{g}(\kappa(\mathbf{L_1})\kappa_{|\mathbf{L}|+|\mathbf{L_2}|})\\&\qquad+(2g-2+||\mathbf{b}||)V_{g}(\kappa(\mathbf{b}))+\sum_{\substack{\mathbf{L}+\mathbf{L^{'}}=\mathbf{b}\\[1pt]||\mathbf{L^{'}}||\geq 2}}\binom{\mathbf{b}}{\mathbf{L}}V_{g}(\kappa(\mathbf{L})\kappa_{|\mathbf{L^{'}}|})
\end{align*}
Hence, Theorem \ref{superHWP2} follows. \qed

\vskip 20pt
\section{Virasoro constraints revisited}
\setcounter{equation}{0}
In this section, we follow the arguments of Section 2 to study properties of generating functions $G^{\Theta}(\hbar,\mathbf{s},\mathbf{t})$ of $\Theta$, $\psi$ and $\kappa$ classes using Theorem \ref{RecursionA} and \ref{RecursionB}. Then we give a new proof of Theorem \ref{Vir}. Usually, recursion formulae can be reformulated in terms of differential operators acting on $G^{\Theta}(\hbar,\mathbf{s},\mathbf{t})$. For example, Proposition 4.1 is equivalent to the following.
\begin{proposition}
$\displaystyle\frac{\partial^2 G^\Theta}{\partial t_0\partial t_1}=\frac{\hbar}{12}\frac{\partial^4 G^\Theta}{\partial t_0^4}+\frac{1}{2}\left(\frac{\partial^2 G^\Theta}{\partial t_0^2}\right)^2.$
\end{proposition}

Denote $\beta_\mathbf{L}=\alpha_{\mathbf{L}}/\mathbf{L!}$, where $\alpha_{\mathbf{L}}$ is defined in Theorem \ref{RecursionB} as $$\mathbf{\alpha_{b}=\mathbf{b}!\sum_{\substack{\mathbf{L}+\mathbf{L^{'}}=\mathbf{b}\\[1pt]\mathbf{L^{'}}\neq \mathbf{0}}}}\frac{(-1)^{||\mathbf{L^{'}}||-1}\mathbf{\alpha_L}}{\mathbf{L}!\mathbf{L^{'}!}(2|\mathbf{L^{'}}|-1)!!},\qquad \mathbf{b}\neq \mathbf{0},$$ with initial data $\mathbf{\alpha}_\mathbf{0}=1.$ For the convience, we introduce the following family of differential operators for $k\geq 0$,
\begin{equation}\begin{aligned}\label{hatV}\hat{V}_k:=&-\frac{(2k+1)!!}{2}\frac{\partial}{\partial t_k}+\frac{1}{2}\sum_{\mathbf{L}}\sum\limits_{j=0}\limits^{\infty}\frac{(2|\mathbf{L}|+2j+2k+1)!!}{(2j-1)!!}\beta_\mathbf{L}\mathbf{s}^{\mathbf{L}}t_j\frac{\partial}{\partial t_{|\mathbf{L}|+j+k}}\\&+\frac{\hbar}{4}\sum\limits_{\mathbf{L}}\sum\limits_{i+j=|\mathbf{L}|+j+k-1}(2i+1)!!(2j+1)!!\beta_\mathbf{L}\mathbf{s}^{\mathbf{L}}\frac{\partial^2}{\partial t_i\partial t_j}+\frac{1}{16}\delta_{k,0},\end{aligned}\end{equation}
where $\mathbf{s}^{\mathbf{L}}=\prod\limits_{i\geq 1}s_j^{L(j)}.$
\begin{theorem}
We have $\hat{V}_k (G^\Theta)=0$ for $k\geq 0$ and $$[\hat{V}_n,\hat{V}_m]=(n-m)\sum\limits_{\mathbf{L}}\beta_\mathbf{L}\mathbf{s}^{\mathbf{L}}V_{n+m+|\mathbf{L}|}.$$
\end{theorem}
\begin{proof}
From the initial data $$\langle \tau_0\rangle^{\Theta}_{1}=\frac{1}{8}\qquad\text{and}\qquad \langle \kappa_1\rangle^{\Theta}_{2}=\frac{3}{128},$$ $\hat{V}_k G=0$ for $k\geq 0$ is just a restatement of Theorem \ref{RecursionB}. The rest part can be checked easily via a direct computation.
$$[\hat{V}_n,\hat{V}_m]=(n-m)\sum\limits_{\mathbf{L}}\beta_\mathbf{L}\mathbf{s}^{\mathbf{L}}V_{n+m+|\mathbf{L}|}.$$
\end{proof}

Let $\gamma_{\mathbf{L}}$ be the inverse to $\beta_{\mathbf{L}}$, $$\gamma_{\mathbf{L}}:=\frac{(-1)^{||\mathbf{L}||}}{\mathbf{L!}(2|\mathbf{L}|-1)!!},$$ and define a new family of differential operators $V_k$ for $g\geq 0$ by
\begin{equation}\begin{aligned}\label{V}V_k:=&-\frac{1}{2}\sum_{\mathbf{L}}\sum\limits_{j=0}\limits^{\infty}(2|\mathbf{L}|+2j+1)!!\gamma_\mathbf{L}\mathbf{s}^{\mathbf{L}}\frac{\partial}{\partial t_{|\mathbf{L}|+j}}+\frac{1}{2}\sum\limits_{j=0}\limits^{\infty}\frac{(2j+2k+1)!!}{(2j-1)!!}t_j\frac{\partial}{\partial t_{j+k}}\\&+\frac{\hbar}{4}\sum\limits_{\mathbf{L}}\sum\limits_{i+j=k-1}(2i+1)!!(2j+1)!!\frac{\partial^2}{\partial t_i\partial t_j}+\frac{1}{16}\delta_{k,0}\end{aligned}\end{equation}

Before we rewrite Theorem \ref{BGW-Theta} in terms of differential
operators, we introduce, following Mulase and Safnuk
\cite{mulase2006mirzakhani}, new variables
$$T_{2i+1}:=\frac{t_i}{(2i+1)!!},\qquad i\geq 0,$$ which transform
$\hat{V}_k$ into
\begin{align*}\hat{V}_k:=&-\frac{1}{2}\frac{\partial}{\partial T_{2k+1}}+\frac{1}{2}\sum_{\mathbf{L}}\sum\limits_{j=0}\limits^{\infty}(2j+1)\beta_\mathbf{L}\mathbf{s}^{\mathbf{L}}T_j\frac{\partial}{\partial T_{2(|\mathbf{L}|+j+k)+1}}\\&+\frac{\hbar}{4}\sum\limits_{\mathbf{L}}\sum\limits_{i+j=|\mathbf{L}|+j+k-1}\beta_\mathbf{L}\mathbf{s}^{\mathbf{L}}\frac{\partial^2}{\partial T_{2i+1}\partial T_{2j+1}}+\frac{1}{16}\delta_{k,0}.\end{align*}
Define operators $J_p$ for $p\in\mathbb{Z}$ by
$$ J_p=\left\{
\begin{aligned}
& (-p)T_{-p}\ \ \text{if}\ p<0,  \\
&\frac{\partial}{\partial T_p}\ \ \text{if}\ p>0.
\end{aligned}
\right.
$$
and $$E_k=\frac{\hbar}{4}\sum\limits_{p\in\mathbb{Z}}J_{2p+1}J_{2(k-p)+1}+\frac{\delta_{k,0}}{16},$$ then $\hat{V}_k$ can be reformulated as $$\hat{V}_k=-\frac{1}{2}J_{2k+1}+\sum\limits_{\mathbf{L}}\beta_{\mathbf{L}}\mathbf{s}^{\mathbf{L}}E_{k+|\mathbf{L}|},$$
It is not difficult to see that$$V_k=\sum\limits_{L}\gamma_{\mathbf{L}}\mathbf{s}^{\mathbf{L}}\hat{V}_{k+|\mathbf{L}|}=-\frac{1}{2}\sum\limits_{\mathbf{L}}\gamma_{\mathbf{L}}\mathbf{s}^{\mathbf{L}}J_{2k+2|\mathbf{L}|+1}+E_k.$$
Then we have
\begin{theorem}For $k\geq 0$, $V_k(G^\Theta)=0$ and $$[V_n,V_m]=(n-m)V_{n+m}.$$
\end{theorem}
\begin{proof}
The first part is a restatement of Thorem 1.4. Now we prove the rest part. Since
\begin{align*}
E_k=&\frac{1}{2}\sum\limits_{j=0}\limits^{\infty}\frac{(2j+2k+1)!!}{(2j-1)!!}t_j\frac{\partial}{\partial t_{j+k}}+\frac{\delta_{k,0}}{16}\\&+\frac{\hbar}{4}\sum\limits_{r+s=k-1}(2r+1)!!(2s+1)!!\frac{\partial^2}{\partial t_r\partial t_s},
\end{align*}
we can check directly that $$[E_n,E_m]=(n-m)E_{n+m},\qquad [J_{2k+1},E_m]=\frac{2k+1}{2}J_{2k+2m+1}.$$
So we have
\begin{align*}
[E_n,E_m]=&[-\frac{1}{2}\sum\limits_{\mathbf{L}}\gamma_{\mathbf{L}}\mathbf{s}^{\mathbf{L}}J_{2k+2|\mathbf{L}|+1}+E_n,-\frac{1}{2}\sum\limits_{\mathbf{L}}\gamma_{\mathbf{L}}\mathbf{s}^{\mathbf{L}}J_{2k+2|\mathbf{L}|+1}+E_m]\\=&-\frac{1}{2}\sum\limits_{\mathbf{L}}\gamma_{\mathbf{L}}\mathbf{s}^{\mathbf{L}}\left([J_{2n+2|\mathbf{L}|+1},E_m]+[E_n,J_{2m+2|\mathbf{L}|+1}]\right)+[E_n,E_m]\\=&-\frac{1}{2}\sum\limits_{\mathbf{L}}\gamma_{\mathbf{L}}\mathbf{s}^{\mathbf{L}}(n-m)J_{2n+2m+2|\mathbf{L}|+1}+(n-m)E_{n+m}\\=&V_{n+m}.
\end{align*}
\end{proof}
Recall that the generating function of intersection numbers involving $\Theta$ and $\psi$ classes is defined in \eqref{ZTheta},
$$Z^{\Theta}(\hbar,t_0,t_1,...)=\mathrm{exp}\underset{g,n,\vec{d}}{\sum}\frac{\hbar^{g-1}}{n!}\int_{\overline{\mathcal{M}}_{g,n}}\Theta_{g,n}\cdot\prod\limits_{i=1}\limits^{n}\psi^{d_i}_i t_{d_i}.$$
Finally, we obtain the relation between $G^{\Theta}(\hbar,\mathbf{s},\mathbf{t})$ and $Z^{\Theta}(\hbar,\mathbf{t})$.

\noindent\textbf{Proof of Theorem \ref{Vir}.}
Denote $$ \widetilde{t}_i=\left\{
\begin{aligned}
& t_i&\text{if}\ i=0,  \\
&t_i-\sum\limits_{|\mathbf{L}|=k}(2|\mathbf{L}|-1)!!\gamma_\mathbf{L}\mathbf{s}^\mathbf{L}\ \ &\text{if}\ i\geq 1.
\end{aligned}
\right.
$$
Then for $k\geq 0$, $V_k$ is transformed to
\begin{align*}
V_k=&-\frac{(2k+1)!!}{2}\frac{\partial}{\partial\widetilde{t}_k}+\frac{1}{2}\sum\limits_{j=0}\limits^{\infty}\frac{(2j+2k+1)!!}{(2j-1)!!}\widetilde{t}_j\frac{\partial}{\partial \widetilde{t}_{j+k}}\\&+\frac{\hbar}{4}\sum\limits_{r+s=k-1}(2r+1)!!(2s+1)!!\frac{\partial^2}{\partial t_i\partial t_j}+\frac{\delta_{k,0}}{16},
\end{align*}
which is indeed the operator obtained by setting $\mathbf{s}=0$ in $\hat{V}_k$ of \eqref{hatV}. Since Virasoro constraints uniquely determine the generating function, then $\forall \mathbf{s}\in N^{\infty}$,
$$G^{\Theta}(\hbar,\mathbf{s},t_0,t_1,...)=Z^{\Theta}(t_0,t_1+p_1(\mathbf{s}),t_2+p_2(\mathbf{s}),...).$$
Therefore, Theorem \ref{Vir} is proved.

\appendix
\section{From recursion formula of the Weil-Petersson volume to Recursion formula of intersection numbers}
Here we show that Stanford-Witten's recursion formula (c.f.
\cite[(D.30)]{stanford2020jt} and \cite[Theorem
2]{norbury2020enumerative}) is equivalent to a recursion formula of
intersection numbers involving $\Theta$, $\psi$ and $\kappa_1$
classes, as noted by Norbury \cite{norbury2020enumerative}.

For the convenience of calculation, we introduce the normalized
volume of \eqref{VTheta}
\begin{equation}\label{Volumedef}\begin{aligned}v^{\Theta}_{g,n}(L_1,...,L_n)&:=\frac{V^{\Theta}_{g,n}(2\pi L_1,...,2\pi L_n)}{(2\pi^2)^d}\qquad(d=g-1)\\&=\frac{1}{d!}\int_{\overline{\mathcal{M}}_{g,n}}\Theta_{g,n}(\kappa_1+\sum_{i=1}^n L_i^2\psi_i)^d\\&=\underset{d_0+d_1+\cdots+d_n=d}{\sum}\langle\kappa_{d_0}\prod\limits_{i=1}\limits^{n}\tau_{d_i} \rangle^{\Theta}_g \frac{\prod\limits_{i=1}\limits^{n}L_i^{2d_i}}{\prod\limits_{i=0}^{n}d_i!},\end{aligned}\end{equation}
where $$\langle\kappa_{d_0}\prod\limits_{i=1}\limits^{n}\tau_{d_i} \rangle^{\Theta}_g=\int_{\overline{\mathcal{M}}_{g,n}}\Theta_{g,n}\kappa_1^{d_0}\prod\limits_{i=1}\limits^{n}\psi_{i}^{d_i}.$$
Notice that $\langle\kappa_{d_0}\prod\limits_{i=1}\limits^{n}\tau_{d_i} \rangle^{\Theta}_g=0$ if $\sum\limits_{i=0}^n d_i\neq g-1$ since deg$\Theta_{g,n}=2g-2+n$. Introduce the kernel fuction \begin{equation}\label{kernelfuction}H(x,y)=\frac{1}{2}\left(\frac{1}{cosh\frac{\pi(x-y)}{2}}-\frac{1}{cosh\frac{\pi(x+y)}{2}}\right),\end{equation}
and the associated kernel fuctions \begin{equation}\label{kernelfuction}D(x,y,z)=H(x,y+z),\qquad R(x,y,z)=\frac{1}{2}\Big(H(x+y,z)+H(x-y,z)\Big).\end{equation}
Then Stanford-Witten's recursion formula for the Weil-Petersson volume of the moduli space of super Riemann surface are given by the following.

\begin{theorem}\textup{\cite{norbury2023new,stanford2020jt}}$v^{\Theta}_{g,n}$ is uniquely determined by $v^{\Theta}_{1,1}(L_1)=\frac{1}{8}$ and the recursion
\begin{equation}\label{Volumerecursion}\begin{aligned}
L_1v^{\Theta}_{g,n}(L_1,L_K)&=\int_{0}^{\infty}\int_{0}^{\infty}xyD(L_1,x,y)P_{g,n+1}(x,y,L_K)dxdy\\&\qquad+\sum\limits_{j=2}\limits^{n}\int_{0}^{\infty}xR(L_1,L_j,x)v^{\Theta}_{g,n-1}(x,L_{K\backslash\{j\}})dx,
\end{aligned}
\end{equation}
where $K=\{2,...,n\}$ and $$P_{g,n+1}(x,y,L_K)=v^{\Theta}_{g-1,n+1}(x,y,L_K)+\sum_{\substack{g_1+g_2=g\\[3pt]I\sqcup J=K}}v^{\Theta}_{g_1,|I|+1}(x,L_I)v^{\Theta}_{g_2,|J|+1}(y,L_J).$$
\end{theorem}

Denote $$h_{2k+1}(t):=\int_{0}^{\infty}\frac{x^{2k+1}}{(2k+1)!!}H(t,x)dx.$$
Then \begin{align*}&\int_{0}^{\infty}\int_{0}^{\infty}\frac{x^{2k+1}}{(2k+1)!!}\frac{y^{2j+1}}{(2j+1)!!}H(t,x+y)dxdy\\&\xlongequal{z=x+y}\int_{0}^{\infty}\int_{0}^{z-x}\frac{x^{2k+1}}{(2k+1)!!}\frac{(z-x)^{2j+1}}{(2j+1)!!}H(t,z)dxdz\\&=\int_{0}^{\infty}\frac{z^{2k+2j+3}}{(2k+2j+3)!!}H(t,z)dz=h_{2k+2j+1}(t).\end{align*}
\begin{lemma}\textup{\cite[Lemma 5.2]{norbury2020enumerative}}
$\displaystyle h_{2k+1}(t)=\sum\limits_{i=0}\limits^{k}\frac{a_{k-i}}{(2k-2i)!}\frac{t^{2i+1}}{(2i+1)!},$ where $a_n$ is determined by $$\frac{1}{\mathrm{cos}x}=\sum\limits_{n=0}\limits^{\infty}a_n\frac{x^{2n}}{(2n)!},$$ with $a_0=1.$
\end{lemma}
\begin{proof}\begin{align*}h_{2k+1}(t)&=\int_{0}^{\infty}\frac{x^{2k+1}}{(2k+1)!!}H(t,x)dx\\&=\int_{0}^{\infty}\frac{x^{2k+1}}{(2k+1)!!}H(x,t)dx\\&=\int_{0}^{\infty}\frac{x^{2k+1}}{(2k+1)!!}\cdot\frac{1}{2}\left(\frac{1}{cosh\frac{\pi(x-y)}{2}}-\frac{1}{cosh\frac{\pi(x+y)}{2}}\right)dx\\&=\frac{1}{2}\times\frac{1}{(2k+1)!}\left[\int_{-t}^{\infty}\frac{(x+t)^{2k+1}}{cosh\frac{\pi x}{2}}dx-\int_{t}^{\infty}\frac{(x-t)^{2k+1}}{cosh\frac{\pi x}{2}}dx\right]\\&=\frac{1}{2}\times\frac{1}{(2k+1)!}\bigg[\int_{0}^{\infty}\frac{(x+t)^{2k+1}-(x-t)^{2k+1}}{cosh\frac{\pi x}{2}}dx\\&\qquad+\int_{-t}^{0}\frac{(x+t)^{2k+1}}{cosh\frac{\pi x}{2}}dx+\int_{-t}^{0}\frac{(x-t)^{2k+1}}{cosh\frac{\pi x}{2}}dx\bigg]\\&=\frac{1}{2}\times\frac{1}{(2k+1)!}\int_{0}^{\infty}\frac{(x+t)^{2k+1}-(x-t)^{2k+1}}{cosh\frac{\pi x}{2}}dx\\&=\frac{1}{(2k+1)!}\sum\limits_{i=0}\limits^{k}\binom{2k+1}{2i+1}t^{2i+1}\int_{0}^{\infty}\frac{x^{2k-2i}}{cosh\frac{\pi x}{2}}dx\\&=\sum\limits_{i=0}\limits^{k}\frac{a_{k-i}}{(2k-2i)!}\frac{t^{2i+1}}{(2i+1)!}.\end{align*}The last equality holds by the property of Laplace transformation.\end{proof}

\begin{theorem}\textup{\cite[Theorem 4]{norbury2020enumerative}}
The recursion formula for Weil-Petersson volume \eqref{Volumerecursion} is equivalent to \eqref{Recursionkappsi}
\end{theorem}
\begin{proof}
We prove $\eqref{Volumerecursion}\Leftrightarrow\eqref{Recursionkappsi}$ by comparing the coefficients of \eqref{Volumerecursion}. By \eqref{Volumedef}, we obtained the following from \eqref{Volumerecursion}
$$\frac{\partial^{2k_1}}{L_1^{2k_1}}\cdots\frac{\partial^{2k_n}}{L_1^{2k_n}}\bigg(v^{\Theta}_{g,n}(L_1,...,L_n)\bigg)|_{L_{i> 0}=0}=\frac{\prod\limits_{i=1}\limits^{n}(2k_i)!}{\prod\limits_{i=0}\limits^{n}k_i!}\langle \kappa_1^{k_0}\prod\limits_{i=1}\limits^{n}\tau_{k_i}\rangle_g^{\Theta},$$

\begin{align*}&\frac{\partial^{2k_1}}{L_1^{2k_1}}\cdots\frac{\partial^{2k_n}}{L_1^{2k_n}}\bigg(\frac{1}{L_1}RHS_1\bigg)|_{L_{i> 0}=0}\\&=\sum\limits_{d_1+d_2+d_0=k_0+k_1-1}\prod\limits_{i=2}\limits^{n}\frac{(2k_i)!}{k_i!}\frac{(2d_1+1)!(2d_2+1)!}{d_0!d_1!d_2!}\times\langle \kappa_1^{d_0}\tau_{d_1}\tau_{d_2}\prod\limits_{i=2}\limits^{n}\tau_{k_i}\rangle_{g-1}^{\Theta}\\&\qquad\times \frac{\partial^{2k_1}}{L_1^{2k_1}}\bigg(\frac{h_{2d_1+2d_2+3}(L_1)}{L_1}\bigg)|_{L_1=0}\\& \ \ +\sum_{\substack{g_1+g_2=g\\[3pt]I\sqcup J=K}}\sum_{\substack{d_0+d_1=g_1-1\\[3pt]d_0^{'}+d_1^{'}=g_2-1}}\prod\limits_{i=2}\limits^{n}\frac{(2k_i)!}{k_i!}\frac{(2d_1+1)!(2d_2+1)!}{d_0!d_1!d_0^{'}!d_1^{'}!} \langle \kappa_1^{d_0}\tau_{d_1}\tau_{k(I)}\rangle_{g_1}^{\Theta}\\&\qquad\times\langle \kappa_1^{d_0^{'}}\tau_{d_1^{'}}\tau_{k(I)}\rangle_{g_2}^{\Theta}\times \frac{\partial^{2k_1}}{L_1^{2k_1}}\bigg(\frac{h_{2d_1+2d_1^{'}+3}(L_1)}{L_1}\bigg)|_{L_1=0},\end{align*}
and
\begin{align*}
&\frac{\partial^{2k_1}}{L_1^{2k_1}}\cdots\frac{\partial^{2k_n}}{L_1^{2k_n}}\bigg(\frac{1}{L_1}RHS_2\bigg)|_{L_{i> 0}=0}\\&=\sum\limits_{j=2}\limits^{n}\sum\limits_{d_0+d_1=k_0+k_1+k_j}\prod\limits_{i\neq 1,j}\frac{(2k_i)!}{k_i!}
\frac{(2d_1+1)!}{d_0!d_1!}\langle \kappa_1^{d_0}\tau_{d_1}\tau_{k(K\backslash \{j\})}\rangle_{g}^{\Theta}\\&\ \times\frac{\partial^{2k_1}}{L_1^{2k_1}}\frac{\partial^{2k_j}}{L_j^{2k_j}}\bigg(\frac{1}{L_1}\int_{0}^{\infty}\frac{x^{2d_1+1}}{(2d_1+1)!}\frac{1}{2}\Big(H(x,L_1+L_j)+H(x,L_1-L_j)\Big)dx\bigg)|_{L_1=0,L_j=0}.\end{align*}
Note that
$$\frac{1}{2}\Big((L_j+L_1)^{2k+1}+(L_1-L_j)^{2k+1}\Big)=L_1\sum\limits_{m}\binom{2k+1}{2m+1}L_1^{2m}L_j^{2(k-m)}.$$
Then by Lemma A.2 and setting $\beta_{\mathbf{L}}=\beta_{(l,0,0,...)}=\beta_l=\frac{2^l}{(2l)!}a_l$ in Lemma 3.1, it follows by the identities above.
\end{proof}

$$ \ \ \ \ $$

\end{document}